\documentclass[12pt]{article}
\usepackage{amsmath}
\usepackage{amssymb}
\usepackage{amsfonts}
\usepackage[labelfont={sc},labelsep=period]{caption}
\usepackage{geometry}
\usepackage{boxedminipage}
\usepackage{color}
\usepackage{comment}

\newtheorem{theorem}{Theorem}

\newtheorem{definition}{Definition}
\newtheorem{lemma}[theorem]{Lemma}

\makeatletter
\def\@begintheorem#1#2{\trivlist
   \item[\hskip \labelsep{\bfseries #1\ #2.}]\itshape}
\makeatother
\newenvironment{proof}[1][Proof]{\noindent\textbf{#1.} }{\ \rule{0.5em}{0.5em}}
\input{tcilatex}

\usepackage{graphicx}
\def\figref#1{\textsc{Figure} \ref{fig:#1}}
\def\figs#1#2#3{%
	\begin{figure}[htb]%
		\centerline{\includegraphics[scale=#3]{#1.eps}}%
		\caption{#2}%
		\label{fig:#1}%
	\end{figure}
}
\def\figstwo#1#2#3#4#5#6{%
	\begin{figure}[htb] \centering%
		\begin{tabular}{cc}
			\includegraphics[scale=#3]{#1a.eps} & \includegraphics[scale=#5]{#1b.eps}
			\\ 
			(a) \ #2 \vspace{0.15in} & (b) \ #4
		\end{tabular}%
		\caption{#6}\label{fig:#1}
	\end{figure}
}

\begin{document}

\title{The $\Lambda $-property of a simple arc.}
\author{J. Ralph Alexander, John E. Wetzel\thanks{%
Corresponding author.}, and Wacharin Wichiramala}
\date{}
\maketitle

\begin{abstract}
In 2006 P. Coulton and Y. Movshovich established an unfamilar but noteworthy
general property of simple, polygonal, open arcs in the plane. \ We give a
new and quite different proof of this property, and we consider a few
generalizations.\medskip


\end{abstract}

\vspace{0.25in}

We give a new proof of a property of open simple arcs proved for polygonal arcs in 2006 by P. Coulton and Y. Movshovich \cite[Th. 5.1]{Coulton Movshovish 2006}:

\begin{theorem}
\textbf{\ }\label{mov th}\textbf{\ }Every simple open arc $\gamma $\ (not a
line segment) has two parallel support lines with the property that the arc
touches one line at two different points $u$ and $w$\ and the other line at
a point $v$ that lies between $u$\ and $w$\ in the parametric ordering of
the points of $\gamma $ (\figref{fig1}).
\end{theorem}

\figs{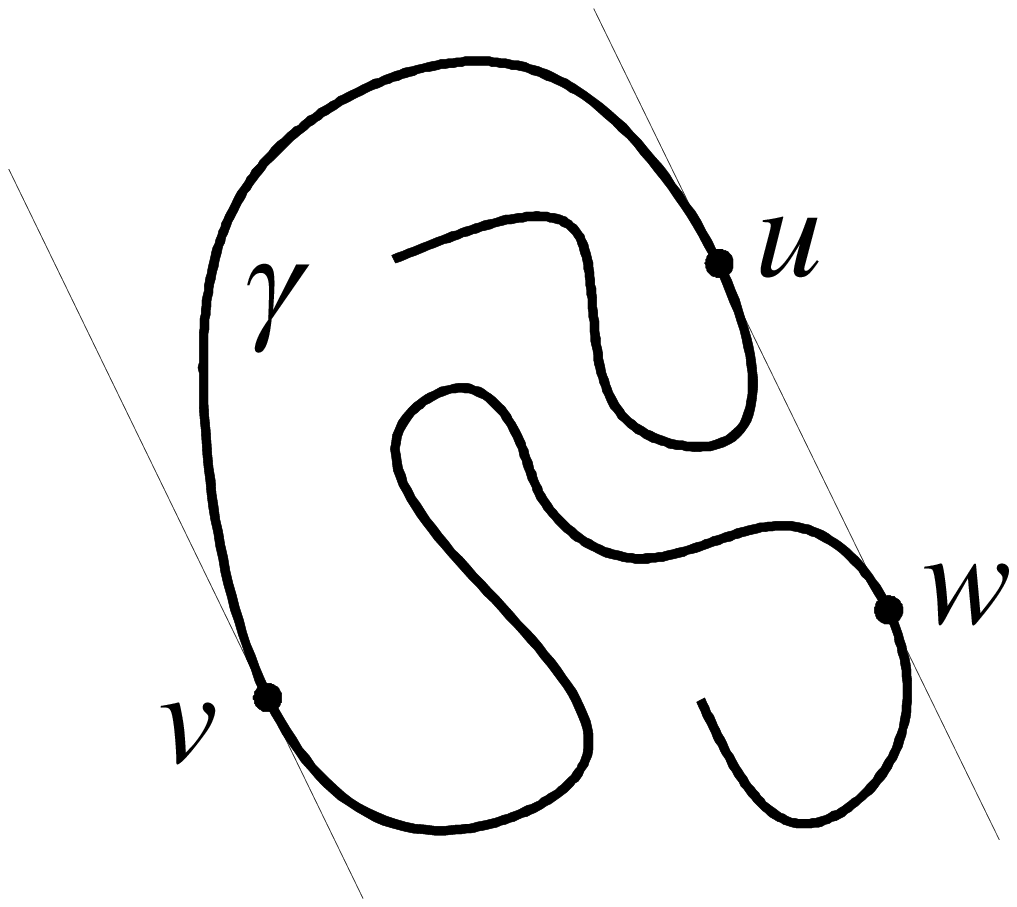}{The $\Lambda $-property.}{0.3}

\noindent They called this the $\QTR{sl}{\Lambda }$\emph{-property} of the
arc.

In the concluding section we extend this theorem to the case in which the
support lines are to form a given angle, and we consider a few related
results.

\section{The parallel case.}

Throughout, by an \emph{open} \emph{simple arc} $\gamma $ we mean the
one-to-one continuous rectifiable image of the real interval $[0,1]$ with $%
\gamma (0)\neq \gamma (1).$ \ For simplicity of language we do not
distinguish notationally between the mapping function $\gamma $ and the
pointset that is its range.

Our proof of Theorem \ref{mov th} begins with a new proof for the case of
polygonal arcs. \ The result for arbitrary arcs then follows from classical
approximation reasoning.

\subsection{Initial considerations.
}

Let $\pi =\langle p_{1}p_{2}\cdots p_{P}\rangle $ be an open, simple,
oriented, polygonal arc, which we suppose is not convex and, in particular,
is not a line segment. \ The points $p_{i}$ are the \emph{nodes} of $\pi $.
\ The closed convex hull $\mathfrak{K}$ of $\pi $ is a polygon. \ Let $%
\partial \mathfrak{K}$ be its boundary, and $\mathcal{V}$ its set of
vertices (a \emph{vertex} of $\mathfrak{K}$ is a point of $\mathfrak{K}$
that does not lie between any two points of $\mathfrak{K}$). \ If $\mathfrak{%
K}$ is a triangle, then one properly chosen side and the parallel line
through the opposite vertex, together with the three suitably labeled
vertices, satisfy the conditions of the Theorem. \ Consequently we may
assume from the start that $\mathfrak{K}$ has $n\geq 4$ vertices.

The segments that connect consecutive vertices on $\partial \mathfrak{K}$
are the \emph{edges} of $\mathfrak{K}.$ \ Each vertex of $\mathfrak{K}$ is a
node of $\pi ,$ but not necessarily conversely (\figref{fig2}). \ 

\figstwo{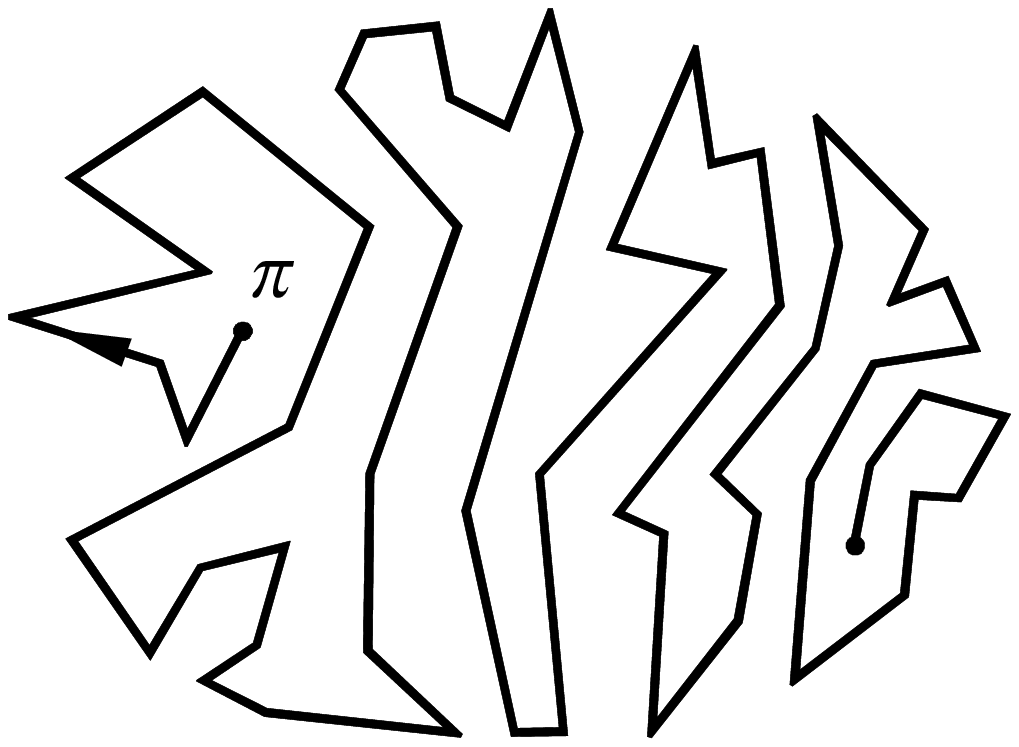}{A path $\pi .$}{0.5}{The convex hull $\mathfrak{K}$
	of $\pi .$}{0.7}{The given path $\pi $ and its convex hull $\mathfrak{K}$.}

A support line of $\mathfrak{K}$ can touch $\mathfrak{K}$ only at a vertex
or along an edge.

\subsubsection{Polygonal paths in $\mathfrak{K.}$}

We are interested in polygonal paths that capture the convexity properties
of $\mathfrak{K}$.

\begin{definition}
A polygonal path $\varpi =\langle w_{1}w_{2}\cdots w_{n}\rangle $ is \emph{%
admissible} if \emph{(1)} it lies in $\mathfrak{K},$ \emph{(2)} it is
simple, and \emph{(3)} the set $\left\{ w_{1},w_{2},\cdots ,w_{n}\right\} $
of its nodes is equal to the vertex set $\mathcal{V}$ of $\mathfrak{K}$.
\end{definition}

To generate such a path, select a vertex of $\mathfrak{K}$ to be the initial
node. \ Once any node but the penultimate one has been located at a vertex
of $\mathfrak{K}$, there are precisely two vertices available for the next
node, because the path must be simple and it must visit each vertex of $%
\mathfrak{K}$ exactly once. \ When the penultimate node has been placed,
just one vertex remains for the last node. \ Consequently:

\begin{lemma}
A strictly convex $n$-gon $\mathfrak{K}$ has $n2^{n-2}$ admissible polygonal
paths each of which has $n$ nodes, and $\mathfrak{K}$ is the convex hull of
each.
\end{lemma}

\subsubsection{The guide path $\protect\varpi $ for $\protect\pi .$}

Among all admissible paths there is just one that visits the vertices of $%
\mathfrak{K}$ in the same sequence that they are visited by the given arc $%
\pi .$ \ We call it the \emph{guide path of} $\pi ,$ and we reserve for it
the notation $\varpi =\langle w_{1}w_{2}\cdots w_{n}\rangle $. \ The guide
path of $\pi $ serves as a \textquotedblleft canonical
form\textquotedblright\ for $\pi .$ \ It is open and simple, and it clearly
is the shortest arc that spans the convex hull of $\pi $ and touches its
vertices in the same sequential order as does $\pi .$

We call $w_{1}$ the \emph{head} and $w_{n}$ the \emph{tail of }$\mathfrak{K}$%
\emph{,} and recalling that $n\geq 4$ we introduce the notation $H=w_{1},$ $%
X=w_{2},$ $Y=w_{n-1},$ and $T=w_{n}.$

We call the oriented line segment $HT$ the \emph{axis} $\xi $ of $\mathfrak{%
K.}$ \ If the axis $\xi $ lies on the boundary of $\mathfrak{K,}$ then $H$
and $T$ are adjacent vertices, and the choices $m$ for the line $HT,$ $n$
the support line of $\mathfrak{K}$ parallel to $m,$ $u=H,$ $w=T,$ and $v$ a
vertex in which $n$ meets $\mathfrak{K}$ show that $\pi $ satisfies the
theorem; so we examine the case in which the axis $\xi $ meets the interior
of $\mathfrak{K.}$ \ Then the points $H$ and $T$ divide the boundary of $%
\mathfrak{K}$ into two convex arcs, one of which passes through the vertex $%
X $. \ We call that arc the \emph{upper boundary }of $\mathfrak{K}$ and the
other arc the \emph{lower boundary. \ }The upper and lower boundary arcs lie
on opposite sides of the axis of $\mathfrak{K.}$ \ We orient each edge $%
[v_{i},v_{i+1}]$ of $\partial \mathfrak{K}$ in the sense from $H$ to $T.$ \ 

The oriented segments $[w_{i},w_{i+1}]$ are the \emph{links} of $\varpi .$ \
These links are of two kinds: \ those that lie on $\partial \mathfrak{K}$
and those that cross $\mathrm{int}(\mathfrak{K})$ and meet the axis $\xi 
\mathfrak{.}$ \ If a link of the guide path $\varpi $ lies on the boundary $%
\partial \mathfrak{K},$ we call it an \emph{edge} link. \ It coincides with
an edge of $\mathfrak{K}$. \ A link of $\varpi $ that is not an edge link is
a \emph{crossing} link.

\subparagraph*{Exemplar.}

It is our intention to follow the general argument using the specific arc
pictured in \figref{fig2}(a) as an exemplar, and for this
purpose we adopt the notations $\pi _{0}$ for this particular arc and $%
\mathfrak{K}_{0}$ for its convex hull, pictured in \figref{fig2}(b). \ Evidently $\pi _{0}$ has $P=60$ nodes, and its convex hull $%
\mathfrak{K}_{0}$ has $16$ vertices. \ The guide path $\varpi _{0}$ for the
exemplar $\pi _{0}$ is shown bold in \figref{fig3}.

\subsubsection{Locales of $\protect\pi .$}

The guide path $\varpi $ partitions the convex hull $\mathfrak{K}$ into $J$
closed convex subregions $\mathfrak{L}_{j}$ called \emph{locales, }indexed $%
\mathfrak{L}_{1},$ $\mathfrak{L}_{2},$ $\dots $ $\mathfrak{L}_{J}$ in order
from $H$ to $T$\ \ Each locale $\mathfrak{L}_{\ell }$ has an edge link $%
B_{j} $ as its base, one or two crossing links as its sides, and a chain $%
C_{j}$ of edge links as its cap. \ The cap $C_{1}$ lies on the upper
boundary of $\mathfrak{K}$ and the base $B_{1}$ on the lower boundary, and
for increasing index the locations alternate, with caps on the upper
boundary and bases on the lower when $j$ is even, and caps on the lower and
bases on the upper when $j$ is odd. \ The situation at $T,$ when $j=J,$
depends on the parity of $J.$

\figs{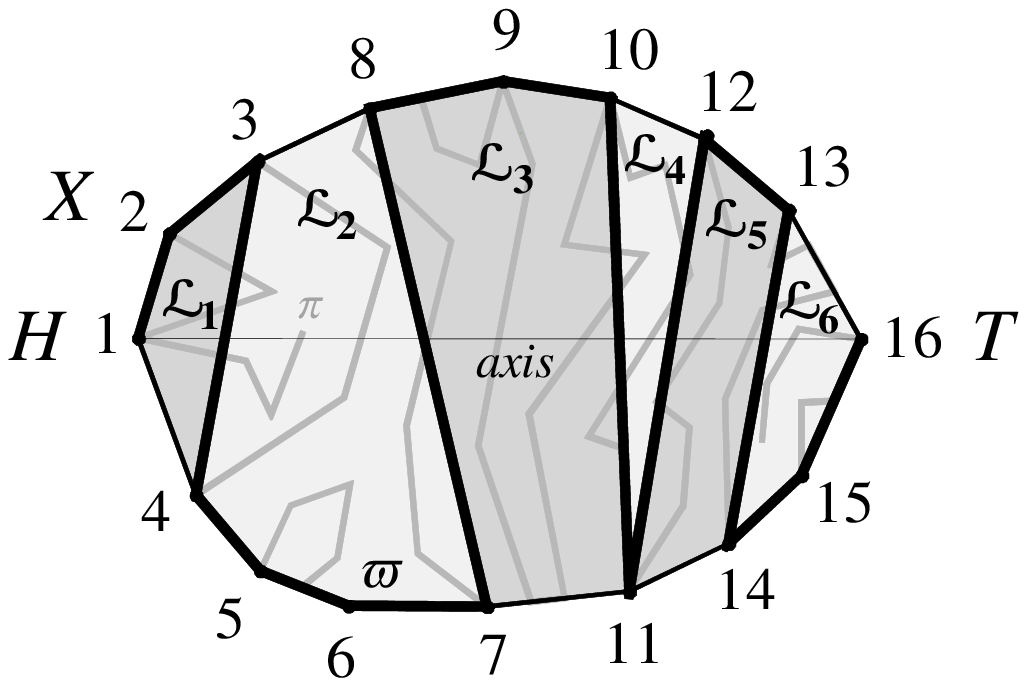}{The guide path $\protect\varpi %
	_{0}$ and locales of $\protect\pi _{0}.$}{0.7}

\subparagraph*{Exemplar.}

\figref{fig3} shows the guide path $\varpi _{0}$ of the
exemplar $\pi _{0}$ and its associated six locales. \ Note that like $C_{4},$
the cap of a locale can be a single point.

\subsubsection{The angles of $\mathfrak{K}$.}

Two important angles, \emph{tilt} and \emph{span, }are associated with each
locale of $\varpi .$\emph{\ \ }In describing them we write $\measuredangle
\left( r_{u},r_{v}\right) $ for the directed angle from a ray $r_{u}$ to the
ray $r_{v}$ in the range $(-180%
{{}^\circ}%
,180%
{{}^\circ}%
],$ positive in the counterclockwise sense and negative in the clockwise
sense.

\subparagraph*{Tilt.}

For each index $j,$ $0\leq j\leq $ $J+1,$\ the \emph{tilt}\ $\tau _{j}$ is
defined to be 
\begin{equation*}
\tau _{j}=\left\{ 
\begin{tabular}{ll}
$\measuredangle \left( \xi ,HX\right) $ & \quad $j=0,$ \\ 
$\measuredangle \left( \xi ,B_{j}\right) $ & \quad $j=1,2,\dots ,J,$ \\ 
$\measuredangle \left( \xi \mathfrak{,}YT\right) $ & \quad $j=J+1,$%
\end{tabular}%
\right.
\end{equation*}%
where $\xi $ is the axis, oriented from $H$ to $T.$ \ (The situation is
illustrated in \figref{fig4}, in which the terminal side of
each tilt angle $\measuredangle \left( \xi \mathfrak{,\tau }_{j}\right) $ is
marked $\tau _{j}$.) \ The tilt of $\mathfrak{L}_{j}$ measures the directed
angle of its base with respect to the horizontal axis $\xi $.

\figstwo{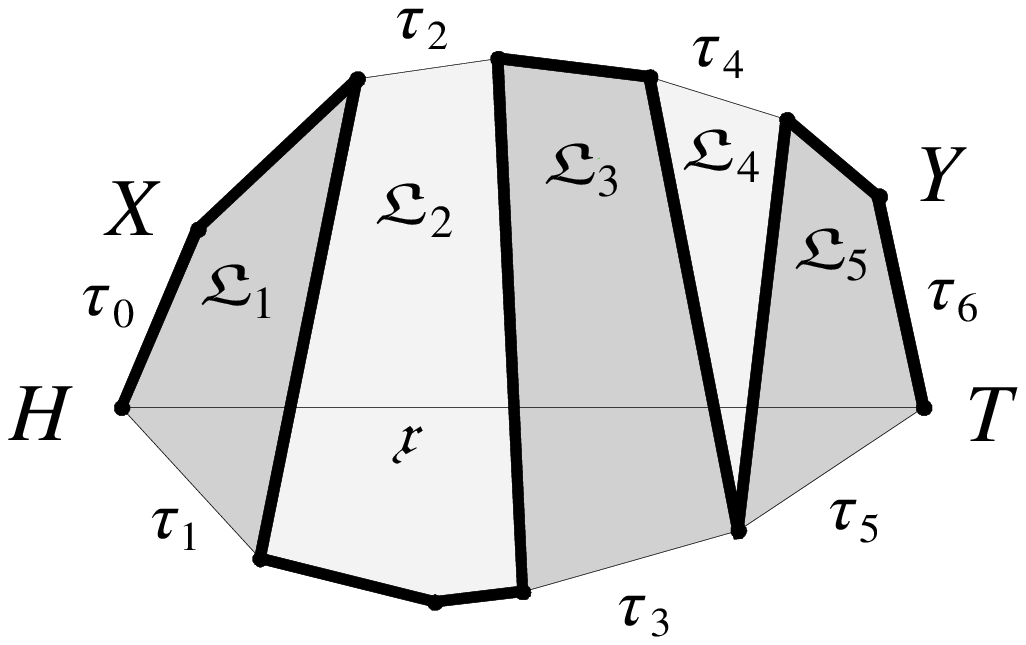}{$\ J=5.$}{0.6}{$J=6.$}{0.6}{The tilt angle.}

\begin{lemma}
\label{monotone}Initial inequalities for $\tau _{0},$ $\tau _{1},$ and $\tau
_{J+1}$ are: $0%
{{}^\circ}%
<\tau _{0}<180%
{{}^\circ}%
,\ -180<\tau _{1}<0%
{{}^\circ}%
,$ and%
\begin{equation*}
\left\{ 
\begin{tabular}{ll}
$0%
{{}^\circ}%
<\tau _{J+1}<180$ & \quad when $J$ is even,\medskip \\ 
$-180%
{{}^\circ}%
<\tau _{J+1}<0%
{{}^\circ}%
$ & \quad when $J$ is odd.%
\end{tabular}%
\right.
\end{equation*}%
The tilts satisfy the monotonicity inequalities 
\begin{equation*}
\begin{tabular}{l}
$\tau _{0}>\tau _{2}>\tau _{4}>\cdots >\left\{ 
\begin{tabular}{ll}
$\tau _{_{J}}$ & \quad if $J$ is even, \\ 
$\tau _{_{J+1}}$ & \quad if $J$ is odd,%
\end{tabular}%
\right. $ \\ 
\thinspace \\ 
$\tau _{1}\,<\tau _{3}<\tau _{5}<\cdots <\left\{ 
\begin{tabular}{ll}
$\tau _{_{J+1}}$ & \quad if $J$ is even, \\ 
$\tau _{_{J}}$ & \quad if $J$ is odd.%
\end{tabular}%
\right. $%
\end{tabular}%
\end{equation*}
\end{lemma}

\begin{proof}
The inequalities for $\tau _{0}$ and $\tau _{1}$ are the result of the
normalization, and the monotonicity claims are consequences of the convexity
of $\mathfrak{K.}$ \ \bigskip
\end{proof}

\subparagraph*{Span.}

The \emph{span} $\delta _{j}$ of the interior locale $\mathfrak{L}_{j}$ is
defined in terms of tilts, as follows. \ For $j=1,2,\dots ,J$:%
\begin{equation}
\mathfrak{\delta }_{j}=\left\{ 
\begin{tabular}{ll}
$\tau _{j+1}-\tau _{j-1}=\measuredangle (B_{j-1},B_{j+1})$ & \quad when $j$
is even, \\ 
$\tau _{j+1}-\tau _{j-1}=\measuredangle (B_{j-1},B_{j+1})$ & \quad when $j$
is odd and $j<J,$ \\ 
$\tau _{J+1}-\tau _{J-1}=\measuredangle \left( B_{J-2},B_{J}\right) $ & 
\quad when $J$ is odd.%
\end{tabular}%
\right.  \label{spacing}
\end{equation}%
When $j$ is even, $\delta _{j}>0,$ reflecting the fact that the support
lines of $\mathfrak{K}$ at points of the cap $C_{j}$ turn counterclockwise
across $C_{j}.$ \ Similarly, when $J$ is odd, $\delta _{j}<0,$ and the
support lines of $\mathfrak{K}$ a points of $C_{j}$ turn clockwise across $%
C_{j}$ (see \figref{fig5}). \ 

\figstwo{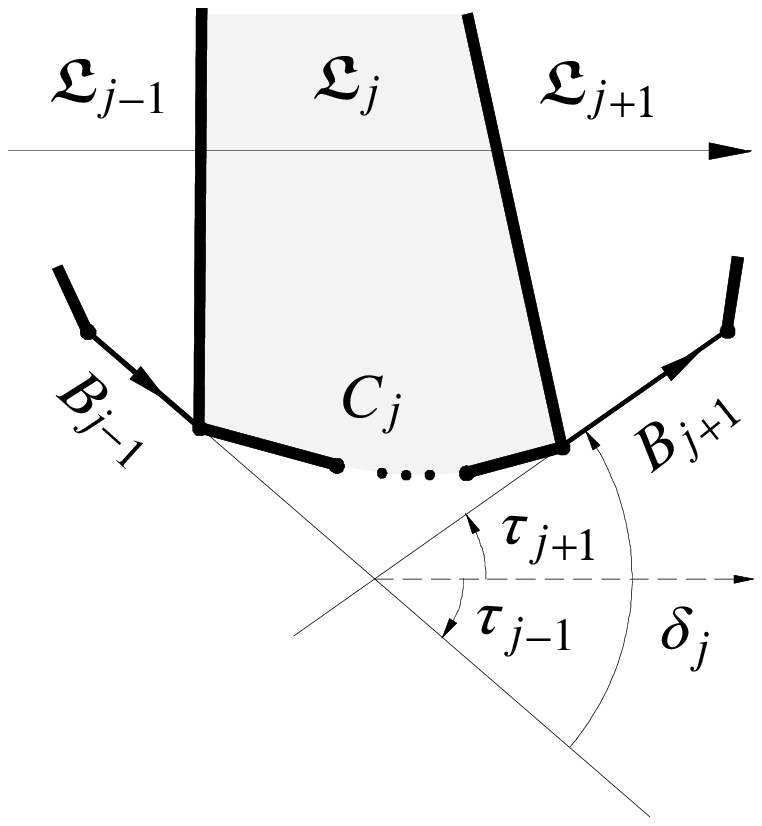}{$j$ even.}{0.6}{$j$ odd.}{0.6}{Span of an interior locale.}
\ \ For every ray $r$ between $B_{j-1}$ and $B_{j+1}$ there is exactly one
line parallel to $r$ that touches the cap $C_{j}$ of $\mathfrak{L}_{j}.$ \
Consequently, the span of $\mathfrak{L}_{j}$ is the range of the angles of
supporting lines that contact $\mathfrak{K}$ on the cap $C_{j}.$

\subsection{The schematic diagram.}

To analyze the data for $\pi $, we employ a schematic diagram built in the
square $[0%
{{}^\circ}%
,360%
{{}^\circ}%
]\times \left( -180%
{{}^\circ}%
,180%
{{}^\circ}%
\right) $ in $\mathbb{R}^{2}$. \ The units on each axis are degrees.

Let $\delta =|\delta _{1}|+|\delta _{2}|+\cdots +|\delta _{J}|,$ so that $%
0<\delta \leq 360%
{{}^\circ}%
.$ \ On the interval $\left[ 0,\delta \right] $ on the $x$-axis, mark off
successive intervals of lengths $|\delta _{1}|,$ $|\delta _{2}|,$ $\dots ,$ $%
|\delta _{J}|,$ and locate a point on the lateral edges of each rectangle 
\begin{equation*}
L_{j}=\left\{ 
\begin{tabular}{ll}
$\left\{ \left( x,y\right) :0\leq x\,<\left\vert \delta _{1}\right\vert ,-180{{}^\circ}<y\leq 180{{}^\circ}\right\} $ & \quad $j=1,\smallskip $ \\ 
$\left\{ \left( x,y\right) :\sum\nolimits_{1}^{j-1}\left\vert \delta
_{i}\right\vert \leq x<\sum\nolimits_{1}^{j}\left\vert \delta
_{i}\right\vert ,\text{ }-180{{}^\circ}<y\leq 180{{}^\circ}\right\} $ & \quad $j=2,3,\ldots ,J.$
\end{tabular}%
\right.
\end{equation*}%
by setting%
\begin{equation*}
T_{j}=\left\{ 
\begin{tabular}{ll}
$\left( 0,\tau _{1}\right) $ & \quad $j=1,\smallskip $ \\ 
$\left( \sum\nolimits_{1}^{j-1}\left\vert \delta _{j}\right\vert ,\tau
_{j}\right) $ & \quad $j=2,3,\ldots,J+1.$%
\end{tabular}%
\right. .
\end{equation*}%
The rectangular region $L_{j}$ encodes the angle data for the locale $%
\mathfrak{L}_{j}$: \ its width is the span $\left\vert \delta
_{j}\right\vert $ of $\mathfrak{L}_{j}$ and the point $T_{j}$ on its left
edge gives the tilt of its base.

In each rectangle $L_{j}$ we insert a horizontal line segment whose left
endpoint is $T_{j}.$ \ Recalling that the width $\left\vert \delta
_{j}\right\vert $ is precisely the spacing $\left\vert \tau _{j-1}-\tau
_{j+1}\right\vert $ (\ref{spacing}), join the consecutive even-indexed
points $T_{j}$ with line segments having slope $-1$ and the consecutive
odd-indexed points $T_{j}$ with line segments having slope $+1,$ as pictured
in (\figref{fig6}). \ Then the horizontal segment at $y=\tau
_{j}$ marks the angle of the base of $\mathfrak{L}_{j},$ and the $\pm 45%
{{}^\circ}%
$ segment codes the range $\delta _{j}$ of angles of the support lines that
touch $\mathfrak{L}_{j}$ on its cap $C_{j}.$

The horizontal and slant line segments assemble themselves into two
continuous piecewise linear paths $\Upsilon $ and $\Phi $ that code the
contact angles $\vartheta $ of the support lines $\ell _{\vartheta }$ that
contact $\mathfrak{K}$ on its upper boundary ($\Upsilon $) and on its lower
boundary ($\Phi $), respectively. \ More precisely, define%
\begin{equation}
\begin{tabular}{l}
$\left. 
\begin{tabular}{l}
$\Upsilon =\langle T_{0}T_{2}T_{4}\cdots T_{J}\rangle $ \\ 
$\Phi =\langle T_{1}T_{3}T_{5}\cdots T_{J+1}\rangle $%
\end{tabular}%
\right\} $\quad $J$ even, \\ 
\\ 
$\left. 
\begin{tabular}{l}
$\Upsilon =\langle T_{0}T_{2}T_{4}\cdots T_{J+1}\rangle $ \\ 
$\Phi =\langle T_{1}T_{3}T_{5}\cdots T_{J}\rangle $%
\end{tabular}%
\right\} $\quad $J$ odd.%
\end{tabular}
\label{fcts}
\end{equation}%
Recalling that $\tau _{0}$ is positive and $\tau _{1}$ is negative, and $%
\tau _{J}$ and $\tau _{J+!}$ always have opposite signs, we see that $%
\Upsilon $ is a continuous, decreasing, piecewise linear function on $\left[
0,\delta \right] $ that falls from $T_{0}$ to $T_{J}$ (or to $T_{J+1}$)$,$
and $\Phi $ is a continuous, increasing piecewise linear function on $\left[
0,\delta \right] $ that rises from $T_{1}$ to $T_{J}$ (or to $T_{J+1}$) (%
\figref{fig6}). \ The fact that these two graphs must meet
and cross constitutes the proof of Theorem \ref{mov th}. 

\figstwo{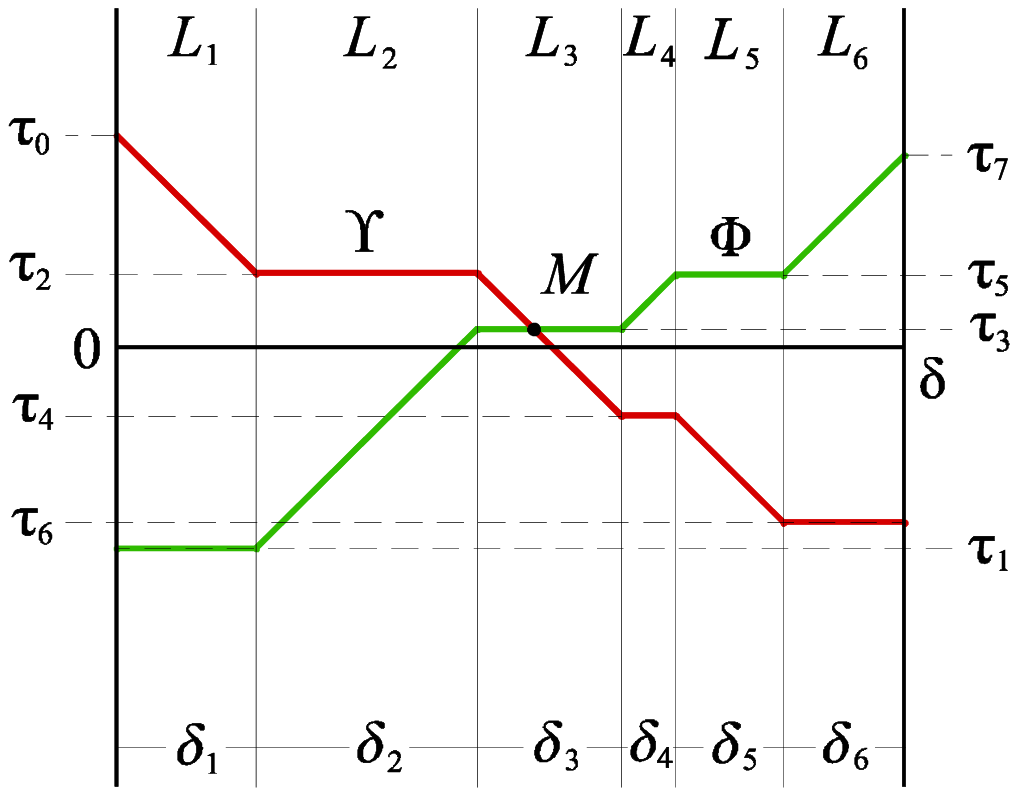}{Schematic with $J$ even.}{0.6}{An example with $J$ odd.}{0.6}{Schematic diagrams.}

We have established the following:

\begin{lemma}
The piecewise linear functions $\Upsilon $ and $\Phi $ defined by (\ref{fcts}%
) meet and cross at precisely one point $M$ inside a strip $L_{j}$ that
represents precisely one locale $\mathfrak{L}_{j}$, and through $M$ there
passes exactly one horizontal segment and exactly one slant segment.
\end{lemma}

\begin{proof}
The decreasing function $\Upsilon $ starts at $0$ above and ends at $\delta $
below the increasing function $\Phi ,$ and the result follows from
continuity. \ 
\end{proof}

\subsection{The proof of Theorem \protect\ref{mov th} for polygonal arcs.}

The point $M$ in which the paths $\Upsilon $ and $\Phi $ meet and cross lies
in exactly one rectangle $L_{j}$ representing precisely one locale $%
\mathfrak{L}_{j},$ and through $M$ pass exactly one horizontal segment and
exactly one $\pm 45%
{{}^\circ}%
$ segment. \ Let $m$ be the (extended) base $B_{j}$ of $\mathfrak{L}_{j}.$ \
There is exactly one line $n$ parallel to $m$ and touches $\mathfrak{K}$ on
the cap $C_{j},$ say at the point $v.$ \ Let $u$ and $w$ be the endpoints of
the base segment $B_{j}$ of $\mathfrak{L}_{j}.$ \ The conditions of Theorem %
\ref{mov th} are satisfied by the lines $m$ and $n$ and the vertices $u,$ $%
v, $ and $w$; and the proof for simple, open, polygonal arcs is completed.

\subparagraph*{Exemplar.}

The drawings involving the exemplar $\pi _{0}$ were made using a CAD
program, which provided the angular data assembled in \textsc{Table} \ref%
{data}.\ 
\begin{table}[htb] \centering%
\begin{tabular}{c}
\begin{tabular}{|c|c|}
\hline
Tilts & Spans \\ \hline
\multicolumn{1}{|l|}{%
\begin{tabular}{l}
$\tau _{0}=73.76%
{{}^\circ}%
$ \\ 
$\tau _{2}=25.70%
{{}^\circ}%
$ \\ 
$\tau _{4}=-23.90%
{{}^\circ}%
$ \\ 
$\tau _{6}=-61.01%
{{}^\circ}%
$ \\ 
$\tau _{1}=-69,98%
{{}^\circ}%
$ \\ 
$\tau _{3}=6.34%
{{}^\circ}%
$ \\ 
$\tau _{5}=25.07%
{{}^\circ}%
$ \\ 
$\tau _{7}=66.82%
{{}^\circ}%
$%
\end{tabular}%
} & \multicolumn{1}{|l|}{%
\begin{tabular}{l}
$\delta _{1}=-48.06%
{{}^\circ}%
$ \\ 
$\delta _{3}=-49.60%
{{}^\circ}%
$ \\ 
$\delta _{5}=-37.10%
{{}^\circ}%
$ \\ 
\\ 
$\delta _{2}=76.32%
{{}^\circ}%
$ \\ 
$\delta _{4}=18.73%
{{}^\circ}%
$ \\ 
$\delta _{6}=41.75%
{{}^\circ}%
$%
\end{tabular}%
} \\ \hline
\end{tabular}%
\medskip%
\end{tabular}%
\caption{Data for the examplar $\pi_{0}$.}\label{data}%
\end{table}%

\figref{fig6}(a) shows the schematic diagram for $\pi _{0}$
drawn from the angle data in \textsc{Table} \ref{data}. \ The crossing point 
$M$ falls in locale $\mathfrak{L}_{3},$ the line $m$ passes through the
vertices $u=w_{7}$ and $w=w_{11}$ and the point $v$ is the vertex $w_{9}$ (\figref{fig7}).

\figs{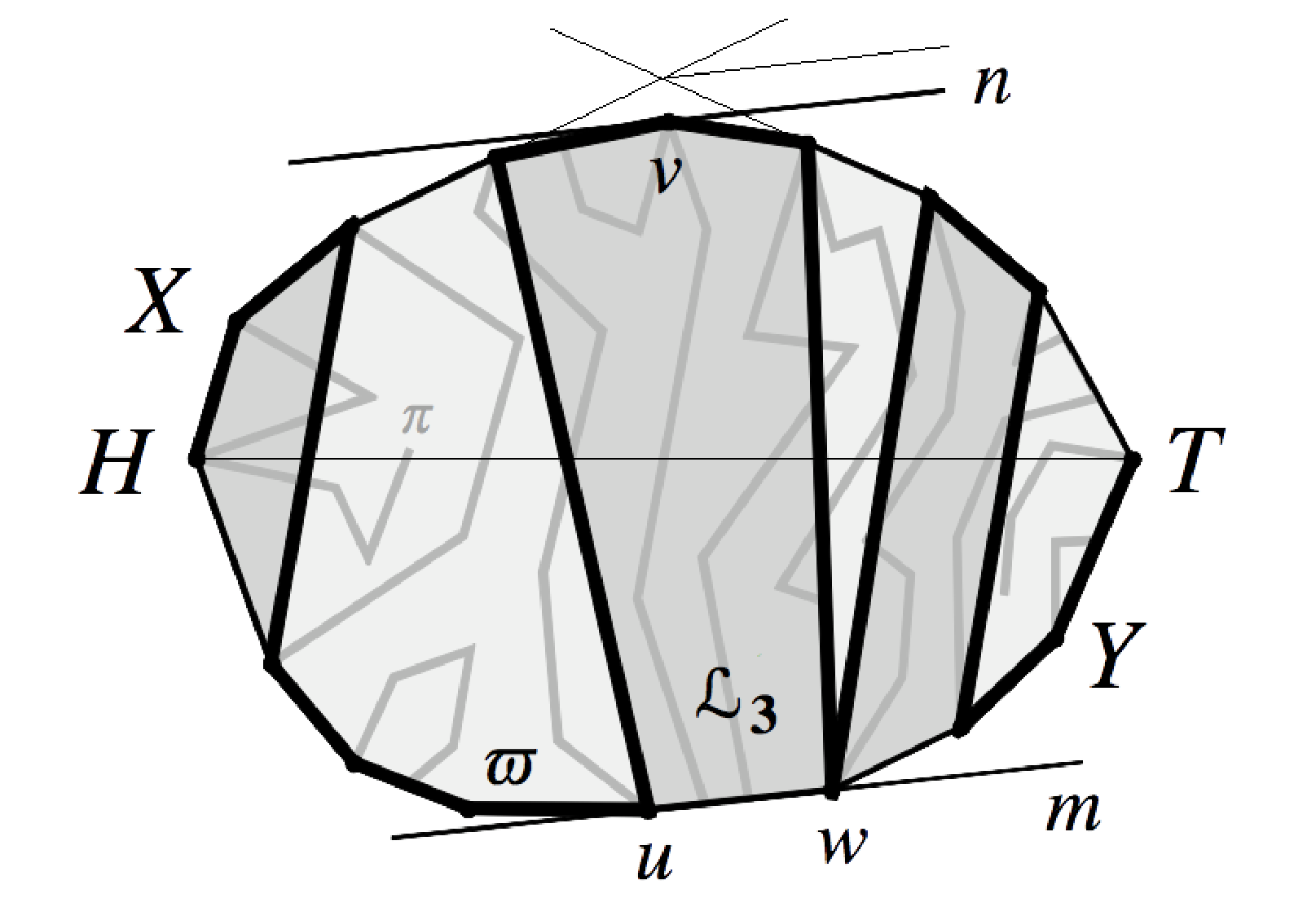}{Support lines for $\protect\pi _{0}.$}{0.22}

\subsection{The extension to rectifiable arcs.}

Classical approximation methods (see, for example, \cite{Mov-Wet 2011} and 
\cite{WW}) can be used to extend the proof for open simple polygonal arcs to
open simple rectifiable arcs, but a little care must be taken to avoid the
confluence of the three points of contact, in which case one can show that
the original curve must have been a straight segment. \ We leave the details
to the interested reader.

\subsection{Closed simple arcs.}

We say that a simple \emph{closed} rectifiable arc $\gamma $ has the $%
\Lambda $ \emph{Property }if it has two parallel support lines $m$ and $n$
with three different points $u,$ $v,$ and $w$ of $\gamma $ \ with $u$ and $w$
on $m$ and $v$ on $n.$ \ Recall that a convex set with non-empty interior is
called \emph{strictly convex} if its boundary contains no straight line
segments (Valentine \cite[p. 94]{val}). \ In this regard we have the
following general theorem:

\figs{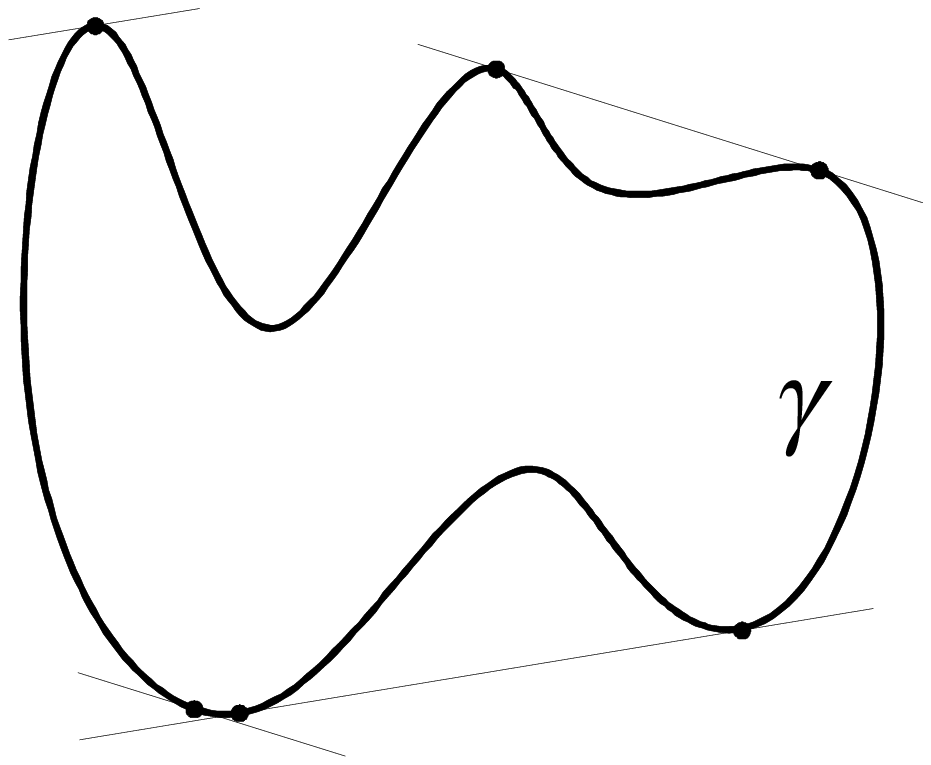}{Closed $\Lambda $ property.}{0.35}

\begin{theorem}
A simple rectifiable arc $\gamma $ has the $\Lambda $ property if and only
if its convex hull is \emph{not} strictly convex.
\end{theorem}

\begin{proof}
If $\gamma $ is open, then it has $\Lambda $ property, and the line segment $%
[u,w]$ lies on the boundary of its convex hull. \ Conversely if the line
segment $[u,w]$ lies in the boundary of its convex hull $\mathfrak{K}$, then
the support line $n$ parallel to $m=\overleftrightarrow{uw}$ meets $%
\mathfrak{K}$ at a point $v,$ and the conditions are satisfied.

Suppose $\gamma $ is closed, and let $\mathfrak{K}$ be its convex hull. \ If 
$\gamma $ has the Movshovich property, then again the line segment $[u,w]$
lies in the boundary $\partial \mathfrak{K}$ of the convex hull, so $%
\mathfrak{K}$ is not strictly convex. \ Finally, if $\mathfrak{K}$ is not
strictly convex, there are points $u$ and $w$ so that the line segment $%
[u,w] $ lies in $\partial \mathfrak{K.}$ \ The line $m=\overleftrightarrow{uw%
}$ is a support line of $\gamma ;$ let $n$ be the support line parallel to $%
m,$ and suppose that $n$ meets $\partial \mathfrak{K}$ at a point $v.$ \
Then $m, $ $n,$ $u,$ $v,$ and $w$ have the desired properties. \ \bigskip
\end{proof}

The uniqueness assertion may fail for simple closed curves, convex or not.

\section{The general case.}

Similar results hold for support lines required to be at prescribed angles.
\ As above we begin with an open, simple, polygonal arc $\pi $ and a given
positive angle $\varphi ,$ and we seek support lines $m$ \ and $n$ and
points $u,$ $v,$ $w$ on $\pi $ so that $\left\vert \measuredangle \left(
m,n\right) \right\vert =\varphi ,$ the points $u$ and $w$ lie on $m,$ the
point $v$ lies on $n,$ and $u,$ $v,$ and $w$ fall in the order $u$-$v$-$w$
in the parametric ordering of $\pi .$ \ Here is the precise result.

\begin{theorem}
\label{theorem}\emph{An open simple, polygonal arc }$\pi $\emph{\ (not a
line segment) and an angle }$\varphi \geq 0$\emph{\ are given. \ Let }$K$%
\emph{\ and }$\varpi $\emph{\ be the convex hull and guide path of }$\pi .$%
\emph{\ \ Partition }$K$\emph{\ into its locales }$\mathcal{L}_{1},$\emph{\ }%
$\mathcal{L}_{2},$\emph{\ }$\dots $\emph{\ }$\mathcal{L}_{J}$\emph{, so that
the tilts }$\tau _{j}$\emph{\ are known. \ Let}%
\begin{align*}
\varphi _{L}& =\tau _{0}-\tau _{1}, \\
\varphi _{R}& =\left\{ 
\begin{tabular}{ll}
$\tau _{J}-\tau _{J+1}$ & \quad $J$ even, \\ 
$\tau _{J+1}-\tau _{J}$ & \quad $J$ odd.%
\end{tabular}%
\right.
\end{align*}%
\emph{Then }$\varphi _{L}>0$\emph{\ and }$\varphi _{R}<0;$\emph{\ we call
the key angles }$\varphi _{L}$ and $\varphi _{R}$ \emph{the left and right} 
\emph{aspect angles} of $\pi ,$ \emph{respectively.}
\end{theorem}

\begin{description}
\item[A.] \emph{If }$\varphi =0,$\emph{\ then there is exactly one pair }$%
(m,n)$\emph{\ of parallel support lines of }$\mathfrak{K}$\emph{\ that has
the } $\Lambda$ \emph{property.}

\item[B.] \emph{If }$0<\varphi \leq \min \left\{ \varphi _{L},\left\vert
\varphi _{R}\right\vert \right\} ,$\emph{\ then there are exactly two pairs }%
$\left( m,n\right) $\emph{\ of support lines of }$\mathfrak{K}$\emph{\ with }%
$\left\vert \measuredangle \left( m,n\right) \right\vert =\varphi $\emph{\
having the }$\Lambda$\emph{\ property, and for one the point }$m\cap n$\emph{%
\ lies to the left of }$H$\emph{\ and for the other the point }$m\cap n$ 
\emph{lies to the right of }$H$\emph{.}

\item[C.] \emph{If }$\min \left\{ \varphi _{L},\left\vert \varphi
_{R}\right\vert \right\} <\varphi \leq \max \left\{ \varphi _{L},\left\vert
\varphi _{R}\right\vert \right\} ,$\emph{\ then there is exactly one pair }$%
\left( m,n\right) $\emph{\ of support lines of }$\mathfrak{K}$\emph{\ with }$%
\left\vert \measuredangle \left( m,n\right) \right\vert =\varphi $\emph{\
having the }$\Lambda $ \emph{property.}

\item[D.] \emph{If }$\max \left\{ \varphi _{L},\left\vert \varphi
_{R}\right\vert \right\} <\varphi ,$\emph{\ then there are no pairs }$\left(
m,n\right) $\emph{\ of support lines that satisfy the desired conditions.}
\end{description}

\begin{proof}
Part (A) is Theorem \ref{mov th}. \ The remaining parts follow immediately
from a modification of the schematic diagram (\figref{fig6}). \ We insert the graph of the difference function $\Delta =\Upsilon -\Phi $
into that diagram to form an extended schematic diagram (\figref{fig9}). \ \ 
\figstwo{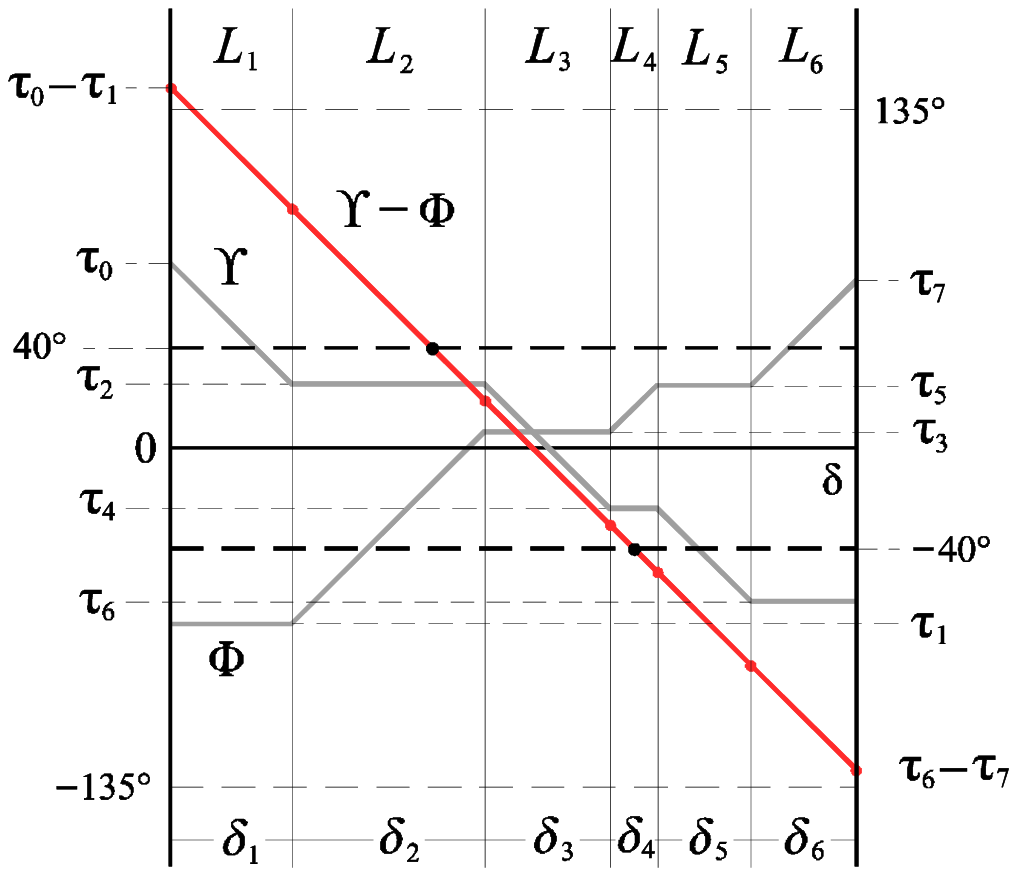}{Schematic diagram for $\pi _{0}$.}{0.6}{The example with $J$ odd.}{0.6}{Augmented schematic diagrams.}
The graph of $\Delta $ is a line segment with slope $-1$ and endpoints $(0%
{{}^\circ}%
,\vartheta _{L})$ and $(\delta ,\vartheta _{R}).$ \ If the line at ordinate $%
\varphi $ meets the graph of $\Delta $ at a point in a rectangle $L_{j},$
then the extended base $B_{j}$ of the associated locale $\mathfrak{L}_{j}$
serves as the support line $m$ and the endpoints $u$ and $w$ of $B_{j}$ as
two of the three desired points; and the second support line $n$ can be
found making the desired angle $\varphi $ with $m$ and contacting the locale 
$\mathfrak{L}_{j}$ at a point $v$ of its cap $C_{j},$ and evidently $u$-$v$-$%
w.$ The three remaining cases of the Theorem reflect the fact that the
horizontal lines at ordinates $\pm \varphi $ meet the graph of $\Delta $ at
a total of two, one, or no points. \ \bigskip
\end{proof}

\subparagraph*{Exemplar.}

The aspect angles of the exemplar $\pi _{0}$ are%
\begin{align*}
\varphi _{L}& =\tau _{0}-\tau _{1}=143.74{{}^\circ}, \\
\left\vert \varphi _{R}\right\vert & =\left\vert \tau _{6}-\tau
_{7}\right\vert =127.83{{}^\circ}.
\end{align*}%
Hence for any angle $\varphi $ with $0<\left\vert \varphi \right\vert \leq $ 
$127.833{{}^\circ}$ there are exactly two pairs of support lines $\left( m,n\right) $ that
satisfy the Movshovich conditions, and for any angle $\varphi $ with $127.833%
{{}^\circ}%
<\varphi \leq $ $143.743%
{{}^\circ}%
$ there is exactly one such pair of support lines. \ For example, for $%
\varphi =40%
{{}^\circ}%
$ there are two support pairs for $\pi _{0}$, found using the dashed lines
at ordinates $\pm 40%
{{}^\circ}%
$ in \figref{fig9}(a); and for $\varphi =135%
{{}^\circ}%
$ there is just one pair of such support lines, found from the lighter
dashed line at ordinate $-135%
{{}^\circ}%
$ in this figure. \ These support lines are pictured in \figref{fig10}.
\figs{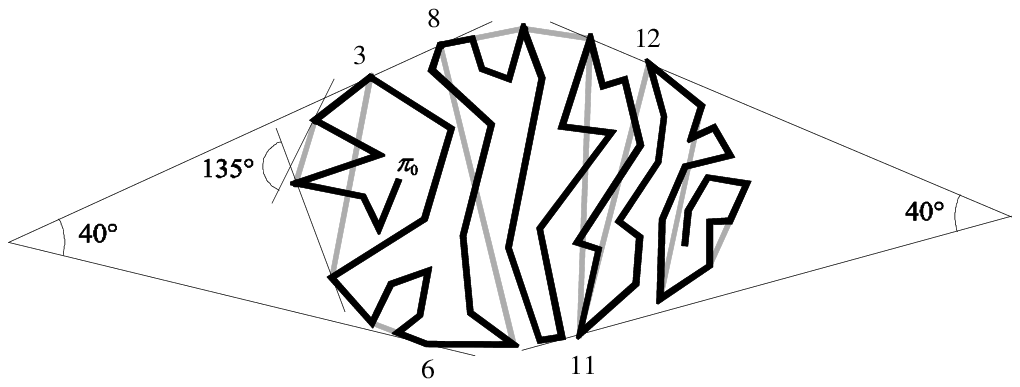}{Support lines at 40${{}^\circ}$ and 135${{}^\circ}.$}{1}

\textbf{Remark.} \ In 1970 one of us conjectured that a 30$%
{{}^\circ}%
$ circular sector having unit radius is a cover for the family $\mathcal{F}$
of all arcs of unit length (see \cite[p. 358]{Wet}), i.e., that this
circular sector contains an isometric copy of every such arc. \ This
conjecture has not been resolved, although there have many partial results,
mostly unpublished. \ If this sector is such a cover, it would be the
smallest convex cover for $\mathcal{F}$ currently known. \ One reason for
our interest in Theorem \ref{theorem} is the hope that it might cast some
light on this conjecture.

J. Ralph Alexander

John E. Wetzel

Department of Mathematics

University of Ilinois at Urbana-Champaign

1409 West Green Street

Urbana, IL 61801

\qquad \emph{johnralph.alexander@gmail.com}

\qquad \emph{jewetz@comcast.net\bigskip }

Wacharin Wichiramala

Department of Mathematics and Computer Science

Faculty of Science

Chulalongkorn University

Bangkok 10330, Thailand

\qquad \emph{wacharin.w@chula.ac.th}


\begin{thebibliography}{9}
\bibitem{Coulton Movshovish 2006} P. Coulton and Y. Movshovich,
\textquotedblleft Besicovich triangles cover unit arcs,\textquotedblright\ 
\emph{Geom Dedicata} 123\textbf{\ }(2006) 79--88.

\bibitem{Mov-Wet 2011} Y. Movshovich and J. Wetzel, \textquotedblleft Escape
paths of Besicovich triangles,\textquotedblright\ \emph{J. Comb. }2 (2011)
413--433.

\bibitem{val} F. A. Valentine, \emph{Convex Sets,} McGraw-Hill, 1964.

\bibitem{Wet} J. Wetzel, \textquotedblleft Fits and
covers,\textquotedblright\ \emph{Math. Mag.,} 76 (2003), 349-363.

\bibitem{WW} J. Wetzel and W. Wichiramala, \textquotedblleft A covering
theorem for families of sets in $\QTR{sl}{R}^{d},$\textquotedblright\ \emph{%
J. Combin.} 1 (2010) 69-76.\bigskip
\end{thebibliography}
\end{document}